\documentclass[12pt]{amsart}
\usepackage{amsfonts}
\usepackage{graphicx}
\usepackage{amsmath}
\usepackage{amssymb}
\input{amssym.def}
\newtheorem{theorem}{Theorem}

\newtheorem{lemma}{Lemma}
\newtheorem{remark}{Remark}

\theoremstyle{remark}

\newcommand{\re}{\text{\rm Re }}
\newcommand{\im}{\text{\rm Im }}

\input epsf
\begin{document}
\title[L\"owner-Kufarev]{Asymptotic conformal welding via L\"owner-Kufarev evolution}
\author[D.~Prokhorov]{Dmitri Prokhorov}

\address{D.~Prokhorov: Department of Mathematics and Mechanics, Saratov State University, Saratov
410012, Russia} \email{ProkhorovDV@info.sgu.ru}

\begin{abstract}
The L\"owner-Kufarev evolution produces asymptotics for mappings onto domains close to the unit
disk $\mathbb D$ or the exterior of $\mathbb D$. We deduce variational formulae which lead to the
asymptotic conformal welding for such domains. The comparison of mappings onto bounded and
unbounded components of the Jordan curve establishes an asymptotic connection between driving
functions in both versions of the L\"owner-Kufarev equation and conformal radii of the two domains.
\end{abstract}
\maketitle

\section{Introduction}

For the unit disk $\mathbb D=\{z: |z|<1\}$ and the complement $\mathbb D^*=\{z: |z|>1\}$ to the
closure of $\mathbb D$, let $f: \mathbb D\to\Omega$ and $F: \mathbb D^*\to\Omega^*$ be conformal
maps where a domain $\Omega$ is bounded by a closed Jordan curve $\Gamma$, and $\Omega^*$ is the
unbounded complementary component of $\Gamma$. The composition $F^{-1}\circ f$ determines a
homeomorphism of the unit circle $\mathbb T=\partial\mathbb D=\partial\mathbb D^*$ which is called
a conformal welding. Suppose that $0\in\Omega$, $f(0)=0$, $f'(0)>0$, and $F(\infty)=\infty$,
$F'(\infty)>0$. We refer to the works \cite{Bishop}, \cite{Jones}, \cite{Markina}, \cite{Takhtajan}
to confirm the recent interest in the conformal welding problems.

An asymptotic conformal welding for domains close to $\mathbb D$ was proposed by the author
\cite{Prokhorov}. It is based on asymptotic formulas for conformal mappings onto these domains. The
bounded version of $f: \mathbb D\to\Omega$ was obtained by Siryk \cite{Siryk}, see also [5, p.
379], and the unbounded version of $F: \mathbb D^*\to\Omega^*$ is given in \cite{Prokhorov}.

{\bf Theorem A.} \cite{Siryk}, \cite{Prokhorov} {\it For the polar coordinates $(r,\psi)$, let
$\Gamma=\partial\Omega=\partial\Omega^*$ have the polar equation $r=r(\psi)=1-\delta(\psi)$,
$0\leq\psi\leq2\pi$, where $\delta(\psi)$ is twice differentiable and
\begin{equation}
|\delta(\psi)|<\epsilon,\;\;\;|\delta'(\psi)|<\epsilon,\;\;\;|\delta''(\psi)|<\epsilon \label{del}.
\end{equation}
Then a function $f: \mathbb D\to\Omega$, $f(0)=0$, $f'(0)>0$, and a function $F: \mathbb
D^*\to\Omega^*$, $F(\infty)=\infty$, $F'(\infty)>0$, have the asymptotic representations}
\begin{equation}
f(z)=z\left(1-\frac{1}{2\pi}\int_0^{2\pi}\delta(\psi)\frac{e^{i\psi}+z}{e^{i\psi}-z}d\psi \right)+
O(\epsilon^2),\;\;\;|z|<1,\;\;\;\epsilon\to+0, \label{Sir}
\end{equation}

\begin{equation}
F(z)=z\left(1-\frac{1}{2\pi}\int_0^{2\pi} \delta(\psi)\frac{z+e^{i\psi}}{z-e^{i\psi}}d\psi \right)+
O(\epsilon^2),\;\;\;|z|>1,\;\;\;\epsilon\to+0. \label{Pr2}
\end{equation}

{\bf Theorem B.} \cite{Prokhorov} {\it Under the conditions of Theorem A and for
$$h(x)=\frac{1}{2\pi}\int_0^{2\pi}(\delta(\psi)-\delta(x))\cot\frac{\psi-x}{2}d\psi,
\;\;\;x\in[0,2\pi],$$ the conformal welding $\sigma=\sigma(s)$ for the domain $\Omega$ bounded by
$\Gamma=\{f(e^{is}): 0\leq s\leq2\pi\}=\{F(e^{i\sigma}): 0\leq\sigma\leq2\pi\}$ satisfies the
asymptotic relation}
$$s+h(s)=\sigma-h(\sigma)+O(\epsilon^2),\;\;\;s\in[0,2\pi],\;\;\;\epsilon\to+0.$$

From the other side, the L\"owner-Kufarev evolution also can produce asymptotics for mappings onto
domains close to $\Omega$ and $\Omega^*$, e.g., for $\Omega=\mathbb D$. The L\"owner equation
\cite{Loewner} is a differential equation obeyed by a family of continuously varying univalent
functions $f(z,t)$, $f(0,t)=0$, from $\mathbb D$ onto a domain with a slit formed by a continuously
increasing arc. The real parameter $t$ characterizes the length of the arc and can be chosen so
that $f(z,t)=e^{-t}z+\dots$, $t\geq0$. Kufarev \cite{Kufarev} and Pommerenke \cite{Pommerenke}
generalized this idea to a wider class of domains. We present here the "decreasing" version of the
L\"owner-Kufarev evolution, see \cite{Gumenyuk} for details of connection between "decreasing" and
"increasing" cases in the L\"owner-Kufarev theory. Given a chain of domains $\Omega(t)$,
$0\in\Omega(t_2)\subset\Omega(t_1)$, $0\leq t_1<t_2<T$, and functions $w=f(z,t):\mathbb
D\to\Omega(t)$ normalized as above, there exist functions $p(z,t)$, $p(\cdot,t)$ are analytic in
$\mathbb D$, $p(z,\cdot)$ are measurable for $0\leq t<T$, and $p$ are from the Carath\'eodory class
which means that
$$p(z,t)=1+p_1(t)z+p_2(t)z^2+\dots,\;\;\;\re p(z,t)>0,\;\;\;z\in\mathbb D,\;\;\;0\leq t<T,$$ such
that
\begin{equation}
\frac{\partial f(z,t)}{\partial t}=-z\frac{\partial f(z,t)}{\partial z}p(z,t) \label{Loe1}
\end{equation}
for $z\in\mathbb D$ and for almost all $t\in[0,T)$, $T$ may be $\infty$. The corresponding
L\"owner-Kufarev equation for the inverse function $z=f^{-1}(w,t):=g(w,t)$ is
\begin{equation}
\frac{\partial g(w,t)}{\partial t}=g(w,t)p(g(w,t),t),\;\;\;w\in\Omega(t),\;\;\;0\leq t<T.
\label{Loe2}
\end{equation}

In case when $\Omega(t)$ are bounded and $\Omega^*(t)$ is the exterior of $\Omega(t)$, let
$w=F(z,t)$ be the unique conformal map from $\mathbb D^*$ onto $\Omega^*(t)$ such that
$F(\infty,t)=\infty$, $F'(\infty,t)>0$, and let $z=G(w,t)$ be the inverse of $F(z,t)$. Let us
normalize the maps so that $F(z,t)=e^{-\tau(t)}z+b_0(t)+b_1(t)z^{-1}+\dots$ as $z\to\infty$ with a
differentiable real function $\tau=\tau(t)$, $\tau(0)=0$, $\tau'(t)>0$, $t\geq0$. Then $F(z,t)$ and
$G(w,t)$ satisfy the L\"owner-Kufarev equations
\begin{equation}
\frac{\partial F(z,t)}{\partial t}=-z\frac{\partial F(z,t)}{\partial
z}q(z,t)\frac{d\tau(t)}{dt},\;\;\;z\in\mathbb D^*,\;\;\;0\leq t<T, \label{Loe3}
\end{equation}
$$\frac{\partial G(w,t)}{\partial
t}=G(w,t)q(G(w,t),t)\frac{d\tau(t)}{dt},\;\;\;w\in\Omega^*(t),\;\;\;0\leq t<T,$$ where $q(\cdot,t)$
are analytic in $\mathbb D^*$, $q(z,\cdot)$ are measurable for $0\leq t<T$, and
$$q(z,t)=1+\frac{q_1(t)}{z}+\frac{q_2(t)}{z^2}+\dots,\;\;\;\re q(z,t)>0,\;\;\; z\in\mathbb D^*,
\;\;\;0\leq t<T.$$

For continuous functions $p(z,\cdot)$ and $q(z,\cdot)$, immediate consequences of (\ref{Loe1}) and
(\ref{Loe3}) with $\Omega(0)=\mathbb D$ are the asymptotic expansions of solutions $f(z,t)$ of
(\ref{Loe1}) and $F(z,t)$ of (\ref{Loe3}). Indeed, since $$f(z,t)=f(z,0)+\left(\frac{\partial
f(z,t)}{\partial t}\right)_{t=0}t+o(t),\;\;|z|<1,\;\;t\to+0,$$ $$zF\left({1\over
z},t(\tau)\right)=zF\left({1\over z},0\right)+z\left(\frac{\partial
F(1/z,t(\tau))}{\partial\tau}\right)_{\tau=0}\tau+o(\tau),\;\;|z|<1,\;\;\tau\to+0,$$ we deduce from
(\ref{Loe1}) and (\ref{Loe3}) that
\begin{equation}
f(z,t)=z-zp(z,0)t+o(t),\;\;\;|z|<1,\;\;\;t\to+0, \label{asy1}
\end{equation}
\begin{equation}
zF\left({1\over z},t\right)=1-q\left({1\over z},0\right)\tau+o(\tau), \;\;\;|z|<1,\;\;\;\tau\to+0.
\label{asy}
\end{equation}

Both equations (\ref{asy1}) and (\ref{asy}) remain true when $f(\cdot,t)$ has a continuous
extension onto $\overline{\mathbb D}=\mathbb D\cup\mathbb T$ and $F(\cdot,t)$ has a continuous
extension onto $\overline{\mathbb D^*}=\mathbb D^*\cup\mathbb T$.

The main result of the article is contained in the following theorem.
\begin{theorem}
Let the driving function $p(\cdot,t)$ from the Carath\'eodory class in (\ref{Loe1}) be $C^2$ in
$\overline{\mathbb D}$ for $0\leq t<T$, $p(z,\cdot)$ be continuous in $[0,T)$ for
$z\in\overline{\mathbb D}$, $p(z,t)$, $p'(z,t)$ and $p''(z,t)$ be bounded in $\overline{\mathbb
D}\times[0,T)$. Then, for solutions $f(z,t)$ to (\ref{Loe1}) with $\Omega(0)=\mathbb D$,
$\Omega(t)=f(\mathbb D,t)$, $\partial\Omega(t)=\Gamma(t)$, and the corresponding functions
$F(\cdot,\tau(t))$, $F(\mathbb T,\tau(t))=f(\mathbb T,t)$, the conformal welding $\varphi: \mathbb
T\to\mathbb T$ of the curve $\Gamma(t)$, $\varphi=\varphi(\tilde{\varphi})$, satisfies the
following relation
\begin{equation}
\varphi=\tilde{\varphi}+2\;\im p(e^{i\tilde{\varphi}},0)t+o(t),\;\;\;t\to+0. \label{main}
\end{equation}
\end{theorem}

Theorem 1 is proved in Section 3 while Section 2 prepares auxiliary results for the proof.

\section{Preliminary statements}

Restrict our considerations to Jordan curves $\Gamma(t)=\partial\Omega(t)=\partial\Omega^*(t)$ of
class $C^{2+\alpha}$, $0<\alpha<1$. This allows us to extend $f:\mathbb D\to\Omega(t)$ and its
derivatives $f'$ and $f''$ continuously onto $\overline{\mathbb D}=\mathbb D\cup\mathbb T$ so that
$f'$ does not vanish there, see, e.g., [11, p. 48]. To provide these properties we require that the
driving function $p(\cdot,t)$ in (\ref{Loe1}) generating $f$ is $C^2$ in $\overline{\mathbb D}$.
The following lemma was proved in \cite{Markina} for $C^{\infty}$-curves. We repeat its formulation
and proof for $C^2$-curves.

\begin{lemma}
Let the function $w(z,t)$ be a solution to the Cauchy problem
\begin{equation}
\frac{dw}{dt}=-wp(w,t),\;\;\;w(z,0)=z,\;\;\;z\in\mathbb D. \label{Loe5}
\end{equation}
If the driving function $p(\cdot,t)$, being from the Carath\'eodory class for almost all $t\geq0$,
is $C^2$ in $\overline{\mathbb D}$ and measurable with respect to $t$, then the boundaries of
$w(\mathbb D,t)\subset\mathbb D$ are $C^3$ for all $t>0$.
\end{lemma}

\begin{proof}
Let $w(z,t)$ be a solution to the Cauchy problem (\ref{Loe5}). It is unique and of class $C^1$ on
$\overline{\mathbb D}$. Differentiate (\ref{Loe5}) inside $\mathbb D$ with respect to $z$ and write
$$\log w'=-\int_0^t(p(w(z,\tau),\tau)+w(z,\tau)p'(w(z,\tau),\tau))d\tau,\;\;\;\log w'(0,t)=-t.$$
The right-hand side is extendable continuously on $\overline{\mathbb D}$. Therefore, $w'$ is $C^1$
and $w$ is $C^2$ on $\overline{\mathbb D}$. Continue analogously and write the formula
$$w''=-w'\int_0^t(2w'(z,\tau)p'(w(z,\tau),\tau)+w(z,\tau)w'(z,\tau)p''(w(z,\tau)\tau))d\tau,$$
which guarantees that $w$ is $C^3$ on $\overline{\mathbb D}$ and completes the proof.
\end{proof}

\begin{lemma}
Let the driving function $p(\cdot,t)$ be from the Carath\'eodory class for almost all $t\in[0,T)$,
$C^2$ in $\overline{\mathbb D}$ and measurable with respect to $t$. Then, for $\Omega(0)=\mathbb
D$, a solution $f(z,t)$ to (\ref{Loe1}) is $C^3$ in $\overline{\mathbb D}$ for all $t\in[0,T)$.
\end{lemma}

\begin{proof}
Let $g(w,t)$ be a solution to (\ref{Loe2}). Choose an arbitrary $s\in[0,T)$ and set
$h(\zeta,t):=g(g^{-1}(\zeta,s),s-t)$, $\zeta\in\mathbb D$, $0\leq t\leq s$. Then $h(\zeta,s)$ is a
solution to the Cauchy problem (\ref{Loe5}) with $p(w,s-t)$ in its right-hand side. By Lemma 1,
$h(\zeta,s)=g^{-1}(\zeta,s)=f(\zeta,s)$ is continuously extendable onto $\overline{\mathbb D}$ and
it is $C^3$ in $\overline{\mathbb D}$ for all $s\in[0,T)$. Take into account that $s$ is
arbitrarily chosen and complete the proof.
\end{proof}

Lemmas 1-2 suppose that $p(\cdot,t)$ is $C^2$ in $\overline{\mathbb D}$. Now we compel continuity
of $p(z,\cdot)$ in (\ref{Loe1}) for $0\leq t<T$ instead of measurability. So the curve
$\Gamma=\Gamma(t)=\partial\Omega(t)$ in Theorem A satisfies the first condition in (\ref{del}) with
$\epsilon=ct$ provided $p(z,0)$ is bounded and $\Gamma(t)$ has the polar equation $r=r(\psi)$,
$0\leq\psi\leq2\pi$. The following lemma provides the latter polar representation.

\begin{lemma}
Let the driving function $p(\cdot,t)$ from the Carath\'eodory class in (\ref{Loe1}) be $C^2$ in
$\overline{\mathbb D}$ for $0\leq t<T$, $p(z,\cdot)$ be continuous in $[0,T)$ for
$z\in\overline{\mathbb D}$ and $p(z ,t)$ and $p'(z,t)$ be bounded in $\overline{\mathbb
D}\times[0,T)$. Then solutions $f(z,t)$ to (\ref{Loe1}) with $\Omega(0)=\mathbb D$ map $\mathbb D$
onto $\Omega(t)$ bounded by $\Gamma(t)$ so that, for $t>0$ small enough and polar coordinates
$(r_t,\psi_t)$, $\Gamma(t)$ has a polar equation $r_t=r_t(\psi_t)$, $0\leq\psi\leq2\pi$.
\end{lemma}

\begin{proof}
The curve $\Gamma(t)=\partial\Omega(t)$  has a polar equation $r_t=r_t(\psi_t)$,
$0\leq\psi\leq2\pi$, if $$\left|\arg\frac{zf'(z,t)}{f(z,t)}\right|\leq\frac{\pi}{2}-\delta,
\;\;\;\delta>0,\;\;\;z\in\overline{\mathbb D}.$$ As in the proof of Lemma 2, choose an arbitrary
$s\in[0,T)$ and set $h(\zeta,t):=g(g^{-1}(\zeta,s),s-t)$, $\zeta\in\mathbb D$, $0\leq t\leq s$,
where $g(w,t)$ is a solution to (\ref{Loe2}). Then $h(\zeta,s)=f(\zeta,s)$ is a solution to the
Cauchy problem (\ref{Loe5}) with $p(w,s-t)$ in its right-hand side. Elementary operations lead us
to the formula $$\arg\frac{\zeta
h'(\zeta,s)}{h(\zeta,s)}=-\im\int_0^sh(\zeta,t)p'(h(\zeta,t),s-t)dt, \;\;\;\zeta\in\mathbb D.$$
Extend $f(\zeta,s)$ and $f'(\zeta,s)$ continuously onto $\overline{\mathbb D}$. The latter formula
implies that $|\arg(\zeta f'(\zeta,s)/f(\zeta,s))|$ is less than $\pi/2$ for
$\zeta\in\overline{\mathbb D}$ and $s$ small enough and completes the proof.
\end{proof}

The proof of Lemma 3 implies that, for $t>0$ small enough,
$$\left|\arg\frac{zf'(z,t)}{f(z,t)}\right|<c_1t,\;\;\;z\in\overline{\mathbb D}.$$

\begin{lemma}
Let the driving function $p(\cdot,t)$ from the Carath\'eodory class in (\ref{Loe1}) be $C^2$ in
$\overline{\mathbb D}$ for $0\leq t<T$, $p(z,\cdot)$ be continuous in $[0,T)$ for
$z\in\overline{\mathbb D}$ and $p(z,t)$ and $p'(z,t)$ be bounded in $\overline{\mathbb
D}\times[0,T)$. Then solutions $f(z,t)$ to (\ref{Loe1}) with $\Omega(0)=\mathbb D$ map $\mathbb D$
onto $\Omega(t)$ bounded by $\Gamma(t)$ so that, for $t>0$ small enough and polar coordinates
$(r_t,\psi_t)$, $\Gamma(t)$ has a polar equation $r_t=r_t(\psi_t):=1-\delta_t(\psi_t)$,
$0\leq\psi\leq2\pi$, where $|\delta_t(\psi_t)|<c_2t$.
\end{lemma}

\begin{proof}
The curve $\Gamma(t)=\partial\Omega(t)$  has a polar equation $r_t=1-\delta_t(\psi_t)$,
$0\leq\psi\leq2\pi$, for $t>0$ small enough. For an arbitrary $s\in[0,T)$, set
$h(\zeta,t):=g(g^{-1}(\zeta,s),s-t)$, $\zeta\in\mathbb D$, $0\leq t\leq s$, where $g(w,t)$ is a
solution to (\ref{Loe2}). Integrate (\ref{Loe5}) for $s>0$ small enough and for $\zeta\in\mathbb T$
and obtain
$$|h(\zeta,s)|=\exp\left\{-\re\int_0^sp(h(\zeta,t),s-t)dt\right\}>
1-\re\int_0^sp(h(\zeta,t),s-t)dt>1-c_2s$$ which completes the proof when $h(\zeta,s)=f(\zeta,s)$ is
continuously extended onto $\overline{\mathbb D}$.
\end{proof}

Now we have to obtain the second condition $|\delta'(\psi)|<\epsilon$ in (\ref{del}).

\begin{lemma}
Let the driving function $p(\cdot,t)$ from the Carath\'eodory class in (\ref{Loe1}) be $C^2$ in
$\overline{\mathbb D}$ for $0\leq t<T$, $p(z,\cdot)$ be continuous in $[0,T)$ for
$z\in\overline{\mathbb D}$ and $p(z,t)$ and $p'(z,t)$ be bounded in $\overline{\mathbb
D}\times[0,T)$. Then solutions $f(z,t)$ to (\ref{Loe1}) with $\Omega(0)=\mathbb D$ map $\mathbb D$
onto $\Omega(t)$ bounded by $\Gamma(t)$ so that, for $t>0$ small enough and polar coordinates
$(r_t,\psi_t)$, $\Gamma(t)$ has a polar equation $r_t=r_t(\psi_t)=1-\delta_t(\psi_t)$,
$0\leq\psi\leq2\pi$, such that $|\delta'_t(\psi_t)|<c_3t$.
\end{lemma}

\begin{proof}
Lemma 3 implies that $\Gamma(t)$ has a polar equation $r_t=1-\delta_t(\psi_t)$. Elementary
reasonings lead to the formula
$$\delta'_t(\psi_t)=
|f(e^{i\varphi},t)|\tan\arg\frac{e^{i\varphi}f'(e^{i\varphi},t)}{f(e^{i\varphi},t)},$$ where
$f(\cdot,t)$ and $f'(\cdot,t)$ are extended continuously onto $\overline{\mathbb D}$ and $\arg
f(e^{i\varphi},t)=\psi_t$. As in the proof of Lemmas 3 and 4, show that
$$|\tan\arg\frac{zf'(z,t)}{f(z,t)}|<\tan(c_1t)<c_3t,\;\;\;z=e^{i\varphi},$$ which completes
the proof.
\end{proof}

Finally, we have to provide the third condition $|\delta''(\psi)|<\epsilon$ in (\ref{del}).

\begin{lemma}
Let the driving function $p(\cdot,t)$ from the Carath\'eodory class in (\ref{Loe1}) be $C^2$ in
$\overline{\mathbb D}$ for $0\leq t<T$, $p(z,\cdot)$ be continuous in $[0,T)$ for
$z\in\overline{\mathbb D}$, $p(z,t)$, $p'(z,t)$ and $p''(z,t)$ be bounded in $\overline{\mathbb
D}\times[0,T)$. Then solutions $f(z,t)$ to (\ref{Loe1}) with $\Omega(0)=\mathbb D$ map $\mathbb D$
onto $\Omega(t)$ bounded by $\Gamma(t)$ so that, for $t>0$ small enough and polar coordinates
$(r_t,\psi_t)$, $\Gamma(t)$ has a polar equation $r_t=r_t(\psi_t)=1-\delta_t(\psi_t)$,
$0\leq\psi\leq2\pi$, such that $|\delta''_t(\psi_t)|<c_4t$.
\end{lemma}

\begin{proof}
Elementary calculations give the formula $$\delta''_t(\psi_t)=-|f(e^{i\varphi},t)|\left(1+2\tan^2
\arg\frac{e^{i\varphi}f'(e^{i\varphi},t)}{f(e^{i\varphi},t)}\right.-$$
$$\left.\frac{(1+\tan^2\arg(e^{i\varphi}f'(e^{i\varphi},t)/f(e^{i\varphi},t)))
\re(1+e^{i\varphi}f''(e^{i\varphi},t)/f'(e^{i\varphi},t))}
{\re(e^{i\varphi}f'(e^{i\varphi},t)/f(e^{i\varphi},t))}\right),$$ where $f(\cdot,t)$, $f'(\cdot,t)$
and $f''(\cdot,t)$ are extended continuously onto $\overline{\mathbb D}$ and $\arg
f(e^{i\varphi},t)=\psi_t$. So it is sufficient to find linear estimates for
$$\re\frac{e^{i\varphi}f'(e^{i\varphi},t)}{f(e^{i\varphi},t)}\;\;\;\text{and}\;\;\;
\re\left(1+\frac{e^{i\varphi}f''(e^{i\varphi},t)}{f'(e^{i\varphi},t)}\right).$$

As in the above Lemmas, for solutions $g(w,t)$ to (\ref{Loe2}) and
$h(\zeta,t)=g(g^{-1}(\zeta,s),s-t)$, $0\leq t\leq s$, $\zeta\in\mathbb D$, we obtain the formulas
$$h'(\zeta,\sigma)=\exp\left\{-\int_0^{\sigma}
(p(h(\zeta,t),s-t)+h(\zeta,t)p'(h(\zeta,t),s-t))dt\right\},\;\;\;0\leq\sigma\leq s,$$
$$\re\frac{\zeta
f'(\zeta,s)}{f(\zeta,s)}=1-\re\int_0^s\zeta h'(\zeta,t)p'(h(\zeta,t),s-t)dt,\;\;\; \zeta\in\mathbb
D,$$ $$\re\frac{\zeta f''(\zeta,s)}{f'(\zeta,s)}=\re\left(\frac{\zeta
f'(\zeta,s)}{f(\zeta,s)}-1-\right.$$ $$\left.\int_0^s(\zeta h'(\zeta,t)p'(h(\zeta,t),s-t)+\zeta
h(\zeta,t)h'(\zeta,t)p''(h(\zeta,t),s-t))dt\right),\;\;\;\zeta\in\mathbb D.$$

This implies that $f'(z,t)$ is bounded and, for $t>0$ small enough,
$$1-c_5t<\re\frac{zf'(z,t)}{f(z,t)}<1+c_5t,\;\;\;
\left|\re\frac{zf''(z,t)}{f'(z,t)}\right|<c_6t,\;\;\;z\in\overline{\mathbb D},$$ which completes
the proof.
\end{proof}

\section{Proof of Theorem 1}

{\it Proof of Theorem 1.} For $s\in[0,T)$, set $h(\zeta,t):=g(g^{-1}(\zeta,s),s-t)$,
$\zeta\in\overline{\mathbb D}$, $0\leq t\leq s$, where $g(w,t)$ is a solution to (\ref{Loe2}).
Equation (\ref{Loe5}) gives after integration that
$$h(\zeta,\sigma)=z\exp\left\{-\int_0^{\sigma}p(h(\zeta,t),s-t)dt\right\},\;\;\;0\leq\sigma
\leq s,\;\;\;\zeta\in\mathbb T,$$ and Lemma 3 allows us to write the polar equation of the curve
$\Gamma(s)=\partial f(\mathbb D,s)$ in the form $r_s=1-\delta_s(\psi_s)$, $0\leq\psi_s\leq2\pi$,
where
\begin{equation}
\delta_s(\psi_s)=1-|f(e^{i\varphi(\psi_s)},s)|=
1-\exp\left\{-\re\int_0^sp(h(e^{i\varphi(\psi_t)},t),s-t)dt\right\} \label{Th1}
\end{equation}
and $\varphi(\psi_s)$ is the inverse function for
\begin{equation}
\psi_s=\varphi-\im\int_0^sp(h(e^{i\varphi},t),s-t)dt,\;\;\; 0\leq\varphi\leq2\pi. \label{Th2}
\end{equation}

Deduce from (\ref{Th1}) that $$\delta_t(\psi_t)=1-\re p(h(e^{i\varphi},0),0)t+O(t^2)=1-\re
p(e^{i\varphi},0)t+O(t^2),\;\;\;t\to+0.$$ Similarly, expansion (\ref{Th2}) gives that
$$\psi_t=\varphi-\im p(h(e^{i\varphi},0),0)t+O(t^2)=\varphi-\im
p(e^{i\varphi},0)t+O(t^2),\;\;\;t\to+0.$$ Equation (\ref{Loe1}) presents the asymptotic expansion
\begin{equation}
f(z,t)=z-zp(z,0)t+O(t^2),\;\;\;t\to+0. \label{Th3}
\end{equation}

Let $F^{-1}(\cdot,\tau(t))\circ f(\cdot,t)$ determine a conformal welding under the conditions of
Theorem 1. According to (\ref{Loe1}), $f(z,t)=e^{-t}z+\dots$, $|z|<1$, and according to
(\ref{Loe3}), $F(z,\tau(t))=e^{-\tau}z+\dots$, $|z|>1$. The Lebedev theorem \cite{Lebedev1}, see
also [7, p. 223], states, that $\tau\leq t$ with the equality sign only in the case when $f(\mathbb
D,t)$ is a disk centered at the origin.

Denote $p^*(z,0)=\overline{p(\overline z,0)}$. Equation (\ref{Pr2}) of Theorem A establishes the
following relations
\begin{equation}
zF\left({1\over z},\tau\right)=1-\frac{1}{2\pi}\int_0^{2\pi}\delta_{\tau}(\psi_{\tau})
\frac{e^{-i\psi}+z}{e^{-i\psi}-z}d\psi+O(\tau^2)= \label{Th4}
\end{equation}
$$1-\left(\frac{1}{2\pi}\int_0^{2\pi}\re p(e^{i\psi},0)
\frac{e^{-i\psi}+z}{e^{-i\psi}-z}d\psi\right)t+O(\tau^2)=$$
$$1-\left(\frac{1}{2\pi}\int_0^{2\pi}\re p^*(e^{-i\psi},0)
\frac{e^{-i\psi}+z}{e^{-i\psi}-z}d\psi\right)t+O(\tau^2)=$$
$$1-p^*(z,0)\tau+O(\tau^2),\;\;\;|z|<1,\;\;\;\tau\to+0.$$

Compare expansion (\ref{Th4}) with (\ref{Sir}) and (\ref{asy}) and observe that
\begin{equation}
p^*(z,0)=q\left({1\over z},0\right),\;\;\;|z|<1. \label{Th5}
\end{equation}

By Lemma 1, $\Gamma(t)=f(\mathbb T,t)=F(\mathbb T,\tau(t))$ are $C^3$-curves. Hence the functions
$q(\cdot,t)$ satisfying (\ref{Loe3}) are $C^2$ extended onto the closure $\overline{\mathbb D^*}$
of $\mathbb D^*$, $q(z,t)$, $q'(z,t)$ and $q''(z,t)$ are bounded in $\overline{\mathbb
D^*}\times[0,t_0]$, $0<t_0<T$.

The welding condition $f(\mathbb T,t)=F(\mathbb T,\tau(t))$, $0\leq t\leq t_0$, gives a source to
obtain an asymptotic representation for $\tau(t)$. Formulas (\ref{asy1}), (\ref{asy}), (\ref{Th3}),
(\ref{Th4}) and (\ref{Th5}) indicate that $\tau(0)=0$ and $\tau(t)=t+o(t)$, $t\to+0$, provided
$\tau(t)$ is differentiable at $t=0$. To confirm differentiability  of $\tau(t)$, write the welding
condition
\begin{equation}
f(e^{i\varphi(\psi_t)},t)=F(e^{i\tilde{\varphi}(\tilde{\psi}_{\tau(t)})},\tau(t)), \label{wel}
\end{equation}
where $\psi_t=\arg f(e^{i\varphi(\psi_t)},t)=\tilde{\psi}_{\tau(t)}=\arg
F(e^{i\tilde{\varphi}(\tilde{\psi}_{\tau(t)})},\tau(t))$.

Equation (\ref{wel}) determines an implicit function $\tau(t)$. Using (\ref{asy1}), (\ref{asy}),
(\ref{Th3}), (\ref{Th4}), (\ref{Th5}), differentiate (\ref{wel}) with respect to $t$ at $t=0$ and
find that $\tau'(0)=1$.

Now we are in a position to prove the final asymptotic relation stated in Theorem 1. Representation
(\ref{Th3}) and also equation (\ref{Th2}) establish a correspondence between points $e^{i\varphi}$
on the unit circle and $f(e^{i\varphi},t)=(1-\delta_t(\psi_t))e^{i\psi_t}$ on the boundary of
$\Omega(t)$, $$\psi_t=\varphi-\im p(e^{i\varphi},0)t+O(t^2),\;\;\;t\to+0.$$ In the same way,
representation (\ref{Th4}) together with (\ref{Th5}) establishes a correspondence between points
$e^{i\tilde{\varphi}}$ on the unit circle and $F(e^{i\tilde{\varphi}},\tau(t))=
(1-\delta_{\tau(t)}(\tilde{\psi}_{\tau(t)}))e^{i\tilde{\psi}_{\tau(t)}}$ on
$\partial\Omega(\tau(t))$, $$\tilde{\psi}_{\tau(t)}=\tilde{\varphi}-\im
q(e^{i\tilde{\varphi}},0)\tau+O(\tau^2)=\tilde{\varphi}+\im
p(e^{i\tilde{\varphi}},0)\tau+O(\tau^2),\;\;\;\tau\to+0,$$ where $\tau(t)=t+o(t)$, $t\to+0$.

Equating $\psi_t=\tilde{\psi}_{\tau(t)}$ we obtain the asymptotic conformal welding (\ref{main})
for domains described in Theorem 1 which completes the proof.

\begin{remark}
It follows from the proof of Theorem 1 that if the Carath\'eodory functions $p(z,t)$, $|z|<1$, in
(\ref{Loe1}) and $q(z,t)$, $|z|>1$, in (\ref{Loe3}) are differentiable in $t$ at $t=0$, then
$\tau(t)$ is twice differentiable at $t=0$, and the term $o(t)$ in formula (\ref{main}) of Theorem
1 can be substituted by $O(t^2)$, $t\to+0$.
\end{remark}

\end{document}